\documentclass[final,leqno,onefignum,onetabnum]{siamltex}

\usepackage{subfigure,upgreek,dsfont,braket,amssymb,amsmath}
\usepackage[small,bf]{caption}
\usepackage{booktabs,dsfont}
\usepackage{algorithm}
\usepackage{float}
\usepackage{graphicx}
\usepackage{algorithmicx}
\usepackage{color,float}

\usepackage{amssymb, amsmath}

\newcommand{\auskommentieren}[1]{}

\newcommand{\beq}{\begin{equation}}
\newcommand{\eeq}{\end{equation}}
\DeclareMathOperator{\graph}{graph}

\title{Error estimate for a finite element approximation of the solution of a linear parabolic equation on a two-dimensional surface} 

\author{Heiko Kr\"oner\thanks{Bereich Optimierung und Approximation, Fachbereich Mathematik, Universit\"at Hamburg, Bundesstrasse 55, 20146 Hamburg, Germany,
{\tt Heiko.Kroener@uni-hamburg.de}}}

\begin{document}
\maketitle
\slugger{mms}{xxxx}{xx}{x}{x--x}

\begin{abstract}
We show that a certain error estimate for a fully discrete finite element approximation
of the solution of the heat equation which is defined in a two-dimensional 
Euclidean domain carries over to the case of a general linear parabolic equation which is defined
on a two-dimensional surface.
\end{abstract}

\section{Introduction}

In many applications it is important to consider PDEs which are defined on surfaces and not in Euclidean space, especially in the case of parabolic equations it is of interest to assume that these surfaces (where the equation is defined) evolve with respect to time in a certain prescribed way.
In \cite{DE07} the so-called evolving surface finite element method (ESFEM) is proposed in order to solve an advection-diffusion equation on an evolving surface, cf. \cite[Sections 1.1 and 1.2]{DE07}.
There are several papers which deal with linear parabolic equations on evolving surfaces, e.g. 
in \cite{KG, DE13} it is shown that classical $L^2$- and $L^{\infty}$- estimates
for a semi-discrete approximation
carry over to ESFEM and in \cite{LMV13, DLM12} Runge-Kutta schemes and backward difference 
schemes are considered; we also mention 
\cite{OR2014, ORX2014}.

In this paper we show in details that an error estimate for a fully discrete
finite element approximation of the
solution of the heat equation which is stated in \cite[Theorem 3.1]{DeckelnickHinze}
for the two and three dimensional Euclidean setting carries over to the case
of a general linear parabolic equation which is defined on a two dimensional surface.
Apart from being of interest by itself we will use our transferred error estimate 
together with \cite{BorukhavaKroener} in a further paper which is joint with Michael Hinze to consider a linear-quadratic 
PDE-restricted optimization problem on moving surfaces for which the motion is a priori 
given.

Our paper is organized as follows. In Section \ref{section1} we state some general facts 
about finite elements on surfaces.
In Section \ref{section2} we formulate the equation under consideration and its 
discretization.
In Section \ref{section3} we state and prove our main result Theorem 
\ref{main_result}.

\section{Finite elements on surfaces}\label{section1}
Let  $\Gamma_0$ be a smooth two-dimensional, embedded, orientable, closed hypersurface in $\mathbb{R}^3$. 
Throughout the paper we choose a fixed finite atlas for $\Gamma_0$. We triangulate $\Gamma_0$ 
by a family $T_h$ of flat triangles with corners (i.e. nodes) lying on $S=\Gamma_0$. We denote the surface of class $C^{0,1}$ given by the union of the triangles $\tau \in T_h$ by $\Gamma_h=S_h$; the union of the corresponding nodes is denoted by $N_h$. Here, $h>0$ denotes a discretization parameter which is related to the triangulation in the following way.
For $\tau \in T$ we define the diameter $\rho(\tau)$  of the smallest disc containing $\tau$, the diameter
 $\sigma(\tau)$ of the largest disc contained in $\tau$ and
\begin{equation}
h = \max_{\tau \in T_h}\rho(\tau), \quad \gamma_h = \min_{\tau \in T_h}\frac{\sigma(\tau)}{h}.
\end{equation}
We assume that the family $(T_h)_{h>0}$ is quasi-uniform, i.e. $\gamma_h \ge \gamma_0 >0$.
We let 
\begin{equation}
V_h= X_h = \{v\in C^0(S_h): v_{|\tau}\ \text{linear for all}\ \tau \in T_h \}
\end{equation}
be the space of continuous piecewise linear finite elements.
Let $N$ be a tubular neighborhood of $S$ in which 
the Euclidean metric of $N$ can be written in the coordinates $(x^0, x)=(x^0, x^i)$ of the tubular neighborhood as
\begin{equation}
\bar g_{\alpha \beta} = (dx^0)^2 + \sigma_{ij}(x)dx^idx^j.
\end{equation} 
Here, $x^0$ denotes the globally (in $N$) defined signed distance to $S$,
$x=(x^i)_{i=1,2}$ local coordinates for $S$ and $\sigma_{ij}=\sigma_{ij}(x)$ the metric of $S$.

For small $h$ we can write $S_h$ as graph (with respect to the coordinates
of the tubular neighborhood) over $S$, i.e.
\begin{equation} \label{30}
S_h = \graph \psi = \{(x^0, x): x^0 = \psi(x), x \in S\}
\end{equation}
where $\psi=\psi_h \in C^{0,1}(S)$ suitable. Note, that 
\begin{equation} \label{31}
|D\psi|_{\sigma}\le c h, \quad |\psi|\le ch^2.
\end{equation}
The induced metric of $S_h$ is given by
\begin{equation}
g_{ij}(\psi(x), x) = \frac{\partial \psi}{\partial x^i}(x) \frac{\partial \psi}{\partial x^j}(x) + \sigma_{ij}(x).
\end{equation}
Hence we have for the metrics, their inverses and their determinants
\begin{equation}
g_{ij}=\sigma_{ij}+O(h^2), \quad 
a^{ij} = \sigma^{ij}+O(h^2) \quad \text{and} \quad
g = \sigma + O(h^2)|\sigma_{ij}\sigma^{ij}|^{\frac{1}{2}}
\end{equation}
where we use summation convention.

 For a function $f:S \rightarrow \mathbb{R}$ we define its lift $\hat f:S_h \rightarrow \mathbb{R}$ to $S_h$ by $f(x) = \hat f(\psi(x), x)$, $x\in S$. For a function $f:S_h \rightarrow \mathbb{R}$ we define its lift $\tilde f:S \rightarrow \mathbb{R}$ to $S$ by $f = \hat{\tilde f}$. This terminus can be obviously extended to subsets.
Let $f \in W^{1,p}(S)$, $g \in W^{1,p^{*}}(S)$, $1\le p \le \infty$ and $p^{*}$
H\"older conjugate of $p$. 
In local coordinates $x=(x^i)$ of $S$ hold
\begin{equation} \label{4}
\int_S \left<D f, D g\right> = \int_S \frac{\partial f}{\partial x^i}\frac{\partial g}{\partial x^j}\sigma^{ij}(x)\sqrt{\sigma(x)}dx^idx^j,
\end{equation}
\begin{equation} \label{5}
\int_{S_h} \left<D \hat f, D \hat  g\right> = \int_{S} \frac{\partial f}
{\partial x^i}\frac{\partial g}{\partial x^j}a^{ij}(\psi(x), x)\sqrt{g(\psi(x), x)}
dx^idx^j,
\end{equation}
\begin{equation}  \label{101}
\int_S \left<D f, D g\right> = \int_{S_h} \left<D \hat f, D \hat  g\right> + 
O(h^2)\|f\|_{W^{1,p}(S)}
\|g\|_{W^{1,p^{*}}(S)},
\end{equation}
and similarly, 
\begin{equation} \label{100}
\int_S f =   \int_{S_h}\hat f+ O(h^2)\|f\|_{L^1(S)}
\end{equation}
where now $f\in L^1(S)$ is sufficient.

The bracket $\left<u,v\right>$ denotes here the scalar product of two tangent vectors $u,v$ (or their covariant counterparts). $\|\cdot \|_{W^{k,p}}$ denotes the usual Sobolev norm, $|\cdot |_{W^{k,p}}=\sum_{|\alpha|=k}\|D^{\alpha}\cdot \|_{L^p}$ and $H^k=W^{k,2}$.

Since the properties and aspects needed to prove a priori error estimates for finite element 
approximations are formulated in terms of integrals 
these observations concerning the transformation behavior of integrals  
essentially imply that the known error estimates from the Euclidean setting 
carry over to the surface case as far as convergence 
of at most quadratic order is concerned. Still, we present details in the following.

\section{The equation and its discretization}\label{section2}
Let $T>0$, $G_T = S \times [0, T]$ and let 
$a^{ij}:G_T \rightarrow T^{2,0}(S)$, $b^i:G_T\rightarrow T^{1,0}(S)$, 
$c: G_T\rightarrow T^{0,0}(S)$ be of class $C^1$
and $a^{ij}(\cdot, t)$, $b^i(\cdot, t)$, $t\in [0,T]$, sections 
of the corresponding tensor bundles, $a^{ij}$ symmetric and positive definite.
We consider the initial value problem
\begin{equation}\label{d1}
\begin{aligned}
 \frac{d}{dt}y -\nabla_i(a^{ij} \nabla_j y) + b^i\nabla_i y + c y =& f \text{ in } G_T, \quad
y(\cdot, 0) = & y_0 
 \end{aligned}
 \end{equation}
where $f\in L^2(0,T; H^1(S))$, $y_0 \in H^2(S)$ and $\nabla$ denotes 
the connection with respect to the induced metric.
Problem (\ref{d1}) has a unique solution $y \in C^0([0, T], H^1(S))\cap
L^2(0, T; H^2(S))$ and we have
\begin{equation}
\max_{0\le t \le T}\|y(t)\|_{H^2(S)}^2 + \int_0^T\|y_t(t)\|_{H^1(S)}^2dt
\le c \left(\|y_0\|_{H^2(S)}^2+\int_0^T\|f(t)\|_{H^1(S)}^2 dt\right).
\end{equation}

Let $0=t_0<t_1<...<t_{N-1}<t_N=T$ be a time grid with $\tau_n=t_n-t_{n-1}$, $n=1, ..., N$ and $\tau=\max_{1\le n\le N}\tau_n$. We set
\begin{equation}
\begin{aligned}
W_{h, \tau} = \{\Phi :& \Gamma_h(0) \times [0, T]\rightarrow \mathbb{R}: \\
& \Phi(\cdot, t)\in X_h \text{ is constant in }t\in (t_{n-1}, t_n), 1 \le n \le N\}
\end{aligned}
\end{equation}
and define the bilinear forms
\begin{equation} \label{8}
a:W^{1,p}(S)\times W^{1,p^{*}}(S)\rightarrow \mathbb{R}, \quad a(u,v) =\int_S \left<Du, Dv\right>+uvdx, 
\end{equation}
\begin{equation} \label{9}
a_h:W^{1,p}(S_h)\times W^{1,p^{*}}(S_h)\rightarrow \mathbb{R}, \quad a_h(u_h,v_h) =\int_{S_h} \left<Du_h, Dv_h\right>
+u_h v_hdx,
\end{equation}
\begin{equation} \label{9_}
a^n_h:W^{1,p}(S_h)\times W^{1,p^{*}}(S_h)\rightarrow \mathbb{R}, \quad  a^n_h(u_h,v_h) =\int_{S_h} \left<Du_h, Dv_h\right>_{\tilde g(t_n)}
+u_h v_hdx,
\end{equation}
\begin{equation} \label{d100}
(Du_h, Dv_h)_{\tilde g(t_n)}=\int_{S_h} \left<Du_h, Dv_h\right>_{\tilde g(t_n)}.
\end{equation}
The last but one equation needs a further definition.
Let $p_1, ..., p_3$ be the middle points of the three sides of $\tau$, 
$\tau \in T_h$, and $v,w \in T^{0,1}(\tau)$ then we define 
\begin{equation} \label{kor1}
\int_{\tau}\left< v,w \right>_{\tilde g (t_n) } = \frac{1}{6}|\tau|\sum_{k=1}^3a^{ij}(\tilde p_k)v_i(p_k)
w_j(p_k)
\end{equation}
where $(a^{ij}(\tilde p_k))$ is a contravariant representation 
with respect to local coordinates $(x^i)$ (belonging to our fixed atlas) in
a neighbourhod of $\tilde p_k$ in $S$ and $(v_i)(p_k)$, $(w_j)(p_k)$ are covariant representations 
with respect to the
orthogonal projections of $\frac{\partial}{\partial x^1}(\tilde p_k)$ and 
$\frac{\partial}{\partial x^2}(\tilde p_k
)$
on $\tau$. 
(Despite similar notation
$\tilde g$ does not refer to a metric.)

We define a discrete operator $G_h:  L^2(S)\rightarrow X_h, v \mapsto G_hv=z_h$ via
\begin{equation} \label{a30}
a_h(z_h, \varphi_h) = \int_{S_h}\hat v\varphi_h \quad \forall \varphi_h \in X_h.
\end{equation}
Furthermore, the brackets $(\cdot, \cdot)$ and $(\cdot, \cdot)_h$ denote the inner products 
of $L^2(S)$ 
and $L^2(S_h)$, respectively, and $\|\cdot \|$, $\|\cdot \|_h$ the 
corresponding norms. The semi-norm associated with the bilinear on the left-hand side of (\ref{d100}) is denoted by $\|\cdot \|_{\tilde g(t_n)}$
We denote the interpolation operator by $I_h$, define $P_h: L^2(\Gamma_0)\rightarrow X_h$ by
\begin{equation}
(\hat z,\phi_h)_h = (P_hz, \phi_h)_h \quad \forall \phi_h \in X_h,
\end{equation} 
let $R_h:H^1(S)\rightarrow X_h$ be defined by
\begin{equation}
a_h(R_h z, \phi_h) =a_h(\hat z, \phi_h)  \quad \forall \phi_h \in X_h
\end{equation}
and $R^n_h:H^1(S)\rightarrow X_h$ by
\begin{equation} \label{n1}
a^n_h(R^n_h z, \phi_h) =a^n_h(\hat z, \phi_h)  \quad \forall \phi_h \in X_h.
\end{equation}
It is well-known that
\begin{equation}
\|\hat z -R_hz\|_{L^2(S_h)} + h \|D(\hat z -R_hz)\|_{L^2(S_h)} \le ch^m \|z\|_{H^m(S)} 
\end{equation}
and
\begin{equation}
\|\hat z -R^n_hz\|_{L^2(S_h)} + h \|D(\hat z -R^n_hz)\|_{L^2(S_h)} \le ch^m \|z\|_{H^m(S)}
\end{equation}
hold for all  $z \in H^m(S)$, $m=1,2$. We conclude for $z \in H^2(S)$ that
\begin{equation}
\begin{aligned}
\|\hat z-R_h z\|_{L^{\infty}(S_h)} &\le \|\hat z-I_h z\|_{L^{\infty}(S_h)} + \|I_h z-R_h z\|_{L^{\infty}(S_h)} \\
&\le c h \|z\|_{H^2(S)} + c h^{-1}\|I_z-R_hz\|_{L^2(S_h)} \le c h\|z\|_{H^2(S_h)}.
\end{aligned}
\end{equation}
There holds 
\begin{equation}
\|\phi_h\|_{L^{\infty}(S_h)} \le \rho(h)\|\phi_h\|_{H^1(S_h)}
\end{equation}
for all $\phi_h \in X_h$ where $\rho(h)=\sqrt{|\log h|}$.

For $Y, \Phi \in W_{h, \tau}$ we let
\begin{equation}
\begin{aligned}
A(Y, \Phi) :=& \sum_{n=1}^N \tau_n (\nabla Y^n, \nabla\Phi^n)_{\tilde g(t_n)} + \sum_{n=2}^N(Y^n-Y^{n-1}, \Phi^n)_h
+(Y^0_{+}, \Phi^0_{+})_h \\
&+\sum_{n=1}^N \tau_n(b^i(t_n)\nabla_i Y^n, \Phi^n)_h + \sum_{n=1}^N\tau_n (c(t_n)Y^n, \Phi^n)_h
\end{aligned}
\end{equation}
where $\Phi^n:=\Phi^n_{-}$, $\Phi^n_{+}=\lim_{s\rightarrow 0+-}\Phi(t_n+s)$. 

Note, that the integrals $(b^i(t_n)\nabla_i Y^n, \Phi^n)_h$ and $(c(t_n)Y^n, \Phi^n)_h$
are defined similarly to (\ref{kor1}) by using a quadrature rule.
 
Our approximation $Y\in W_{h, \tau}$ of the solution $y$ of (\ref{d1}) is obtained by the following discontinuous 
Galerkin scheme:
\begin{equation} \label{n10}
A(Y, \Phi) = \sum_{n=1}^N\int_{t_{n-1}}^{t_n}(\widehat{y}, \Phi^n)_h + (\hat y_0, \Phi^0_{+})_h \quad 
\forall \phi \in W_{h, \tau}.
\end{equation}
The above solution will be denoted by $Y=G_{h, \tau}(y)$. 

\section{The error estimate}\label{section3}
We have the following uniform error estimate.
\begin{theorem} \label{main_result}
We have
\begin{equation} \label{n50}
\max_{1\le n\le N}\|\widehat y(\cdot, t_n)-Y^n\|_{L^{\infty}(S_h)}
\le c\rho(h)(h+\sqrt{\tau})(\|y_0\|_{H^2(S)}+\|f\|_{L^2(0, T; H^1(S)}).
\end{equation}
\end{theorem}
\begin{proof}
We adapt the proof of \cite[Theorem 3.1]{DeckelnickHinze}. Take $\Phi \in W_{h, \tau}$, multiply (\ref{d1})
by $\widetilde{\Phi^n}$ and integrate over $\Gamma_0 \times (t_{n-1}, t_n)$. Abbreviating $y^n:=y(\cdot, t_n)$
we have
\begin{equation} \label{n2}
\begin{aligned}
(y^n-&y^{n-1}, \widetilde{\Phi^n}) + \int_{t_{n-1}}^{t_n}\int_{\Gamma_0}a^{ij}D_i y D_j \widetilde{\Phi^n}
+b^iD_i y \widetilde{\Phi^n}+cy\widetilde{\Phi^n} dxdt \\
&
= \int_{t_{n-1}}^{t_n}(f, \widetilde{\Phi^n})dt.
\end{aligned}
\end{equation}
Next, let us introduce $\tilde Y\in W_{h, \tau}$ by
\begin{equation}
\tilde Y(\cdot, t):= R_h^ny^n, \quad t \in (t_{n-1}, t_n), 1 \le n \le N.
\end{equation}
Using (\ref{n1}) and (\ref{n2}) we derive 
\begin{equation}
\begin{aligned}
A(\tilde Y, \Phi) =& \sum_{n=1}^N\tau_n(\nabla R_h^ny^n, \nabla \Phi^n)_{\tilde g(t_n)}
+ \sum_{n=2}^N(R_h^ny^n-R_h^{n-1}y^{n-1}, \Phi^n)_h \\
&+(R_h^ny ^1, \Phi^0_+)_h \\
&+\sum_{n=1}^N \tau_n(b^i(t_n)\nabla_i \widehat{y^n}, \Phi^n)_h + \sum_{n=1}^N\tau_n (c(t_n)\widehat{y^n}, \Phi^n)_h \\
=& \sum_{n=1}^N\tau_n (\widehat{y^n}, \Phi^n)_h + \sum_{n=1}^N\tau_n(\nabla \widehat{y^n}, \nabla \Phi^n)_{\tilde g(t_n)}-\sum_{n=1}^N\tau_n(R_h^ny^n, \Phi^n)_h \\
&+ \sum_{n=2}^N(R_h^ny^n-R_h^{n-1}y^{n-1}, \Phi^n)_h
+(R_h^1y ^1, \Phi^0_+)_h \\
&+\sum_{n=1}^N \tau_n(b^i(t_n)\nabla_i \widehat{y^n}, \Phi^n)_h + \sum_{n=1}^N\tau_n (c(t_n)\widehat{y^n}, \Phi^n)_h\\
=& \sum_{n=1}^N\tau_n (\widehat{y^n}, \Phi^n)_h + \sum_{n=1}^N\tau_n(\nabla \widehat{y^n}, \nabla \Phi^n)_{\tilde g(t_n)}-\sum_{n=1}^N\tau_n(R_h^ny^n, \Phi^n)_h \\
&+ \sum_{n=2}^N(\widehat{y^n}, \Phi^n)_h+\sum_{n=2}^N(\nabla \widehat{y^n}, \nabla \Phi^n)_{\tilde g(t_n)}
-\sum_{n=2}^N(\nabla R_h^ny ^n, \nabla \Phi^n)_{\tilde g(t_n)} \\
&- \sum_{n=2}^N(\widehat{y^{n-1}}, \Phi^n)_h-\sum_{n=2}^N(\nabla \widehat{y^{n-1}}, \nabla \Phi^n)_{\tilde g(t_n)} \\
&+\sum_{n=2}^N(\nabla R_h^{n-1}y ^{n-1}, \nabla \Phi^n)_{\tilde g(t_n)} +(R_h^1y ^1, \Phi^0_+)_h \\
&+\sum_{n=1}^N \tau_n(b^i(t_n)\nabla_i \widehat{y^n}, \Phi^n)_h + \sum_{n=1}^N\tau_n (c(t_n)\widehat{y^n}, \Phi^n)_h
\end{aligned}
\end{equation}
\begin{equation}
\begin{aligned}
=& \sum_{n=1}^N \tau_n\left(\nabla \widehat{y^n}-\frac{1}{\tau_n}\int_{t_{n-1}}^{t_n}\nabla \hat y dt, \nabla \Phi ^n\right)_{\tilde g(t_n)}\\
&+\sum_{n=1}^N\int_{t_{n-1}}^{t_n}\int_{\Gamma_0}a^{ij}\nabla_i y \nabla_j \widetilde{\Phi^n} dxdt\\
&+ \sum_{n=1}^N\left( \int_{t_{n-1}}^{t_n} \nabla \hat y dt, \nabla \Phi^n\right)_{\tilde g(t_n)}
-\sum_{n=1}^N\int_{t_{n-1}}^{t_n}\int_{\Gamma_0}a^{ij}\nabla_i y \nabla_j \widetilde{\Phi^n} dxdt \\
&+ \sum_{n=1}^N\tau_n \left(b^i(t_n)\left(\nabla_i \widehat{y^n}-\frac{1}{\tau_n}\int_{t_{n-1}}^{t_n}
\nabla_i \hat y dt \right), \Phi^n\right)_h\\
& + \sum_{n=1}^N\int_{t_{n-1}}^{t_n}\int_{\Gamma_0}b^i\nabla_i y \widetilde{\Phi^n}dxdt \\
&+ \sum_{n=1}^N\left(b^i(t_n) \int_{t_{n-1}}^{t_n} \nabla_i \hat y dt, \Phi^n\right)_h
-\sum_{n=1}^N\int_{t_{n-1}}^{t_n}\int_{\Gamma_0}b^i\nabla_i y \widetilde{\Phi^n} dxdt \\
&+ \sum_{n=1}^N\tau_n \left(c(t_n)\left( \widehat{y^n}-\frac{1}{\tau_n}\int_{t_{n-1}}^{t_n}
\hat y dt \right), \Phi^n\right)_h + \sum_{n=1}^N\int_{t_{n-1}}^{t_n}\int_{\Gamma_0}c y \widetilde{\Phi^n}dxdt \\
&+ \sum_{n=1}^N\left(c(t_n) \int_{t_{n-1}}^{t_n} \hat y dt, \Phi^n\right)_h
- \sum_{n=1}^N\int_{t_{n-1}}^{t_n}\int_{\Gamma_0}cy \widetilde{\Phi^n}dxdt\\
&+ \sum_{n=1}^N\tau_n(\widehat{y^n}-R_h^ny^n, \Phi^n)_h \\
& + \sum_{n=2}^N\int_{t_{n-1}}^{t_n}(f(t), \widetilde{\Phi^n})dt\\
& - \sum_{n=2}^N \int_{t_{n-1}}^{t_n}\int_{\Gamma_0}a^{ij}\nabla_i y, \nabla_j \widetilde{\Phi^n}-b^i\nabla_iy \widetilde{\Phi^n}-cy\widetilde{\Phi^n}dxdt  \\
&+\sum_{n=2}^N (\widehat{y^n}-\widehat{y^{n-1}}, \Phi^n)_h-\sum_{n=2}^N(y^n-y^{n-1}, \widetilde{\Phi^n}) \\
&+ \sum_{n=2}^N(R_h^ny^n-\widehat{y^n}, \Phi^n)_h-\sum_{n=2}^N(R_h^{n-1}y^{n-1}-\widehat{y^{n-1}}, \Phi^n)_h \\
&+ (R_h^1y^1, \Phi^0_+)_h \\
=& J_1 + ... + J_{19} \\
=&J_1+J_5+J_9+J_{13}+J_{17}+J_{18}+J_{19}\\
& + \int_{t_0}^{t_1}\int_{\Gamma_0}a^{ij}\nabla_i y \nabla_j \widetilde{\Phi^n}+b^i\nabla_iy \widetilde{\Phi^n}+cy\widetilde{\Phi^n} dxdt \\
&+ \sum_{n=1}^N\left( \int_{t_{n-1}}^{t_n} \nabla \hat y dt, \nabla \Phi^n\right)_{\tilde g(t_n)}
+ \sum_{n=1}^N\left(b^i(t_n) \int_{t_{n-1}}^{t_n} \nabla_i \hat y dt, \Phi^n\right)_h\\
\end{aligned}
\end{equation}
\begin{equation}
\begin{aligned}
&+ \sum_{n=1}^N\left(c(t_n) \int_{t_{n-1}}^{t_n} \hat y dt, \Phi^n\right)_h \\
&-\sum_{n=1}^N\int_{t_{n-1}}^{t_n}\int_{\Gamma_0}a^{ij}\nabla_i y \nabla_j \widetilde{\Phi^n}+b^i\nabla_iy \widetilde{\Phi^n}+cy\widetilde{\Phi^n} dxdt \\
&+ \sum_{n=2}^N \int_{t_{n-1}}^{t_n}(f(t), \widetilde{\Phi^n})dt \\
&+\sum_{n=2}^N (\widehat{y^n}-\widehat{y^{n-1}}, \Phi^n)_h-\sum_{n=2}^N(y^n-y^{n-1}, \widetilde{\Phi^n})\\
=&  \sum_{n=1}^N \int_{t_{n-1}}^{t_n}(f(t), \widetilde{\Phi^n})dt+(y_0, \widetilde{\Phi^1})  \\
&+J_1+J_5+J_9+J_{13}+J_{17}+J_{18}+ (R_h^1y^1-\widehat{y^1}, \Phi^0_+)_h\\
&+ (\widehat{y^1}, \Phi^0_+)_h-(y^1, \widetilde{\Phi^0_+})\\
&+ \sum_{n=1}^N\left( \int_{t_{n-1}}^{t_n} \nabla \hat y dt, \nabla \Phi^n\right)_{\tilde g(t_n)}
+ \sum_{n=1}^N\left(b^i(t_n) \int_{t_{n-1}}^{t_n} \nabla_i \hat y dt, \Phi^n\right)_h\\
&+ \sum_{n=1}^N\left(c(t_n) \int_{t_{n-1}}^{t_n} \hat y dt, \Phi^n\right)_h \\
& -\sum_{n=1}^N\int_{t_{n-1}}^{t_n}\int_{\Gamma_0}a^{ij}\nabla_i y \nabla_j \widetilde{\Phi^n} 
+b^i\nabla_iy \widetilde{\Phi^n}+cy\widetilde{\Phi^n} dxdt \\
&+\sum_{n=2}^N (\widehat{y^n}-\widehat{y^{n-1}}, \Phi^n)_h-\sum_{n=2}^N(y^n-y^{n-1}, \widetilde{\Phi^n})
\end{aligned}
\end{equation}
To be in accordance with \cite{DeckelnickHinze} we denote
$r_1(\Phi):=J_1+J_5+J_9$, $r_2(\Phi):=J_{13}$, $r_3(\Phi):=J_{17}+J_{18}$, $r_4(\Phi)=(R_h^1y^1-\widehat{y^1}, \Phi^0_+)_h$ and $r(\Phi):= \sum_{i=1}^4 r_i(\Phi)$.

As a consequence, the error $E:= \tilde Y-Y$ satisfies
\begin{equation} \label{n4}
\begin{aligned}
A(E, \Phi) =& r(\Phi)\\
 & +\sum_{n=1}^N \int_{t_{n-1}}^{t_n}(f(t), \widetilde{\Phi^n})dt+(y_0, \widetilde{\Phi^1})  \\
 &-\sum_{n=1}^N\int_{t_{n-1}}^{t_n}(\widehat{f(t)}, \Phi^n)_h + (\hat y_0, \Phi^0_{+})_h \\
&+ (\widehat{y^1}, \Phi^0_+)_h-(y^1, \widetilde{\Phi^0_+})\\
&+ \sum_{n=1}^N\left( \int_{t_{n-1}}^{t_n} \nabla \hat y dt, \nabla \Phi^n\right)_{\tilde g(t_n)}
+ \sum_{n=1}^N\left(b^i(t_n) \int_{t_{n-1}}^{t_n} \nabla_i \hat y dt, \Phi^n\right)_h\\
\end{aligned}
\end{equation}
\begin{equation}
\begin{aligned}
&+ \sum_{n=1}^N\left(c(t_n) \int_{t_{n-1}}^{t_n} \hat y dt, \Phi^n\right)_h \\
& -\sum_{n=1}^N\int_{t_{n-1}}^{t_n}\int_{\Gamma_0}a^{ij}\nabla_i y \nabla_j \widetilde{\Phi^n} 
+b^i\nabla_iy \widetilde{\Phi^n}+cy\widetilde{\Phi^n} dxdt \\
&+\sum_{n=2}^N (\widehat{y^n}-\widehat{y^{n-1}}, \Phi^n)_h-\sum_{n=2}^N(y^n-y^{n-1}, \widetilde{\Phi^n}) \\
=& r(\Phi) + L_1+ ... + L_{12}
\end{aligned}
\end{equation}
Let $R_n, S_n$, $n=1, ..., N$, be real numbers then we use the (only formally reasonable) notation
\begin{equation}
 \left(\sum_{n=1}^NR_n\right)*S_n:=\sum_{n=1}^N R_nS_n.
\end{equation}
In this sense we can consider $A(E, \phi)*S_n$ where we assume $A(E, \phi)$ being written as a sum from $1$ to $N$ with summands  given by the terms on the right-hand side of (\ref{n4}) numbered by the time index of the test function, here $\Phi^1 = \Phi^0_+$ is considered to have time index $1$.

The idea is as follows. In order to show the error estimate we want to use that $|c|$ is large and that $c$ has the correct sign. Both conditions are not fulfilled in general. Therefore we define
\begin{equation} \label{n11}
\bar Y^n := e^{\mu t_n}Y^n
\end{equation}
where $\mu=1$ (we write $\mu$ instead of 1 to make the influence of this factor clear).

We derive a scheme which is fulfilled by $\bar Y$. This new scheme has compared to the scheme (\ref{n10})
an additional summand on the left-hand side and a modified right-hand side as will become clear in the following.

There holds
\begin{equation}
e^{\mu t_{n-1}}= e^{\mu t_n}- \mu e^{\mu t_n}\tau_n + \frac{1}{2}\mu^2e^{\mu\xi_n}\tau_n^2
\end{equation}
with some $\xi_n \in [t_{n-1}, t_n]$ and hence
\begin{equation}
\begin{aligned}
A(\bar Y, \Phi) =& A(Y, \Phi) * e^{\mu t_n} + \sum_{n=2}^N(e^{\mu t_n}-e^{\mu t_{n-1}})(Y^{n-1}, \Phi^n)_h  \\
=& A(Y, \Phi) * e^{\mu t_n} + \sum_{n=2}^N\mu (\tau_n - \frac{1}{2}\mu \tau_n^2e^{\mu(\xi_n-t_n)})(\bar Y^n, \Phi^n)_h\\
&- \sum_{n=2}^N\mu(\tau_n - \frac{1}{2}\mu \tau_n^2e^{\mu(\xi_n-t_n)})(\bar R^n, \Phi^n)_h \\
& - \sum_{n=2}^N\mu(\tau_n(1-e^{\mu \tau_n}) - \frac{1}{2}\mu \tau_n^2e^{\mu(\xi_n-t_n)}(1-e^{\mu\tau_n}))(\bar Y^{n-1}, \Phi^n)_h \\
=& A(Y, \Phi) * e^{\mu t_n} + M_1+ ... +M_3
\end{aligned}
\end{equation}
where $\bar R^n = \bar Y^n-\bar Y ^{n-1}$ and we remark that $M_1, ..., M_3$ depend on $Y, \Phi$.
We set $\overline{\tilde Y}^n = e^{\mu t_n} \tilde Y^n$, $\bar E^n = \overline{\tilde Y}^n-\bar Y^n$ and have
\begin{equation} \label{n12}
A(\bar E, \Phi) -M_1...-M_3 = A(E, \Phi) * e^{\mu t_n} 
\end{equation}
where 
\begin{equation}
\begin{aligned}
M_1+M_2 =& \sum_{n=2}^N \mu(\tau_n - \frac{1}{2}\mu \tau_n^2e^{\mu(\xi_n-t_n)})(\bar E^n, \Phi^n)_hÊ\\
&-\sum_{n=2}^N \mu(\tau_n-\frac{1}{2}\mu \tau_n^2e^{\mu(\xi_n-t_n)})( \bar E^n-\bar E^{n-1}, \Phi^n)_h
\end{aligned}
\end{equation}
and
\begin{equation}
|M_3| \le c\sum_{n=2}^N\tau_n^2|(\bar E^{n-1}, \Phi^n)_h|.
\end{equation}

Let us fix $l \in \{2, ..., N\}$ and define $\Phi \in W_{h, \tau}$ by
\begin{equation}
\Phi^n:= 
\begin{cases}
0, \quad n=1 \text{ or } n>l \\
\frac{\bar E^n-\bar E^{n-1}}{\tau_n}, \quad 2 \le n \le l.
\end{cases}
\end{equation}
We insert $\Phi$ in (\ref{n12}) and estimate the resulting left-hand side from below and the right-hand side from above. We begin with the first estimate.
We have
\begin{equation}
\begin{aligned}
A(\bar E, \Phi) =& \sum_{n=2}^l\tau_n\left(\nabla \bar E^n, \frac{\nabla \bar E^n-\nabla \bar E^{n-1}}{\tau_n}\right)_{\tilde g(t_n)} \\
&+ \sum_{n=2}^l\left(\bar E^n-\bar E^{n-1}, \frac{\bar E^n-\bar E^{n-1}}{\tau_n}\right)_h \\
& + \sum_{n=2}^l \tau_n\left(b^i(t_n)\nabla_i\bar E^n, \frac{\bar E^n-\bar E^{n-1}}{\tau_n}\right)_h \\
&+ \sum_{n=2}^l \tau_n\left(c(t_n)\bar E^n, \frac{\bar E^n-\bar E^{n-1}}{\tau_n}\right)_h \\
=& K_1 + ... + K_4.
\end{aligned}
\end{equation}
We have
\begin{equation}
\begin{aligned}
\sum_{n=2}^l &\left(\nabla \bar E^n-\nabla \bar E^{n-1}, \nabla \bar E^n-\nabla \bar E^{n-1}\right)_{\tilde g(\tau_n)} \\
=&
\sum_{n=2}^l(\nabla \bar E^n, \nabla \bar E^{n})_{\tilde g(t_n)}  -2 \sum_{n=2}^l(\nabla \bar E^n, \nabla \bar E^{n-1})_{\tilde g(t_n)}
\\
&+\sum_{n=2}^l(\nabla \bar E^{n-1}, \nabla \bar E^{n-1})_{\tilde g(t_n)}
+(\nabla \bar E^1, \nabla \bar E^1)_{\tilde g(t_2)} \\
\le& 
\sum_{n=2}^l(\nabla \bar E^n, \nabla \bar E^{n})_{\tilde g(t_n)} (2+c \tau_n) -2 \sum_{n=2}^l(\nabla \bar E^n, \nabla \bar E^{n-1})_{\tilde g(t_n)}\\
&+(\nabla \bar E^1, \nabla \bar E^1)_{\tilde g(t_2)}-(\nabla \bar E^l, \nabla \bar E^l)_{\tilde g(t_l)} \\
\le& 2 K_1 +c \sum_{n=1}^l(\nabla \bar E^n, \nabla \bar E^n)_{\tilde g(t_n)}\tau_n +(\nabla \bar E^1, \nabla \bar E^1)_{\tilde g(t_1)} 
-(\nabla \bar E^l, \nabla \bar E^l)_{\tilde g(t_n)}
\end{aligned}
\end{equation}
and hence
\begin{equation}
\begin{aligned}
K_1 \ge& \frac{1}{2}\sum_{n=2}^l\|\nabla \bar E^n-\nabla \bar E^{n-1}\|^2_{\tilde g(\tau_n)}
+\|\nabla \bar E^l\|^2_{\tilde g(t_n)}-\|\nabla \bar E^1\|^2_{\tilde g(t_1)}\\
& -c \sum_{n=1}^l\|\nabla \bar E^n\|^2_{\tilde g(t_n)}\tau_n.
\end{aligned}
\end{equation}
We write
\begin{equation}
K_2 = \sum_{n=2}^l\tau_n \left\|\frac{\bar E^n-\bar E^{n-1}}{\tau_n}\right\|_h^2.
\end{equation}
We have
\begin{equation}
\begin{aligned}
|K_3| \le & \sum_{n=2}^l\frac{1}{4}\tau_n \left\|\frac{\bar E^n-\bar E^{n-1}}{\tau_n}
\right\|_h^2 + 
\sum_{n=2}^l c \left\|\nabla \bar E\right\|_{\tilde g(t_n)}^2 \tau_n \\
|K_4|\le & \sum_{n=2}^l \frac{1}{4}\tau_n \left\|\frac{\bar E^n-\bar E^{n-1}}{\tau_n}\right\|_h^2
+ c \sum_{n=2}^l\tau_n \left\|\bar E^n\right\|_h^2.
\end{aligned}
\end{equation}
Furthermore, there holds
\begin{equation}
\begin{aligned}
-M_1 =& \sum_{n=2}^l \mu(\tau_n + \frac{1}{2}\mu \tau_n^2e^{\mu(\xi_n-\tau_n)})(\bar E^n, \Phi^n)_h\\
=& \sum_{n=2}^l \mu(1 + \frac{1}{2}\mu \tau_ne^{\mu(\xi_n-\tau_n)})(\bar E^n, \bar E^n-\bar E^{n-1})_h\\
=& \sum_{n=2}^l \mu(1 + \frac{1}{2}\mu \tau_ne^{\mu(\xi_n-\tau_n)})\\
& \{(\bar E^n, \bar E^n)_h-(\bar E^n-\bar E^{n-1}, \bar E^{n-1})_h
-(\bar E^{n-1}, \bar E^{n-1})_h\}\\
\ge & \frac{\mu}{2}(\bar E^l, \bar E^l)_h-\mu(\bar E^1, \bar E^1)_h
-\sum_{n=1}^{l-1}\mu c \tau  e^{\mu \tau}(\bar E^n, \bar E^n)_h \\
& +\frac{1}{8}\sum_{n=2}^l\tau_n \left\|\frac{\bar E^n-\bar E^{n-1}}{\tau_n}\right\|_h^2
\\
-M_2 \ge& -\frac{K_2}{4}.
\end{aligned}
\end{equation}

Hence the left-hand side of (\ref{n12}) can be estimated from below by
\begin{equation}
\begin{aligned}
& \frac{1}{2}\sum_{n=2}^l\|\nabla \bar E^n-\nabla \bar E^{n-1}\|^2_{\tilde g(\tau_n)}
+\frac{1}{2}\|\nabla \bar E^l\|^2_{\tilde g(t_l)}-\|\nabla \bar E^1\|^2_{\tilde g(t_1)}\\
&-\mu \|\bar E^1\|_h^2 -c \sum_{n=1}^{l-1}\|\nabla \bar E^n\|^2_{\tilde g(t_n)}\tau_n + \sum_{n=2}^l\frac{\tau_n}{8} \left\|\frac{\bar E^n-\bar E^{n-1}}{\tau_n}\right\|_h^2 \\
&+ \frac{\mu}{4}\|\bar E^l\|^2_h- \sum_{n=1}^{l-1}\mu c \tau e^{\mu \tau}\|\bar E^n\|^2_h.
\end{aligned}
\end{equation}
In the following we estimate the right-hand side of (\ref{n12}) from above and refer to \cite[Theorem 3.1]{DeckelnickHinze} for more details concerning the estimate of $r_i(\Phi)$, $i=1, 2,3$.

We have
\begin{equation}
\begin{aligned}
|r_1(\Phi)| \le& \frac{1}{8}\sum_{n=2}^l\|\nabla \bar E^n-\nabla \bar E^{n-1}\|_h^2 + c \tau \int_0^T\|\nabla y_t\|_{g(t)}^2 dt, \\
& + c \tau \|y_t\|^2_{L^2(0,T; L^2(\Gamma_0))} + \sum_{n=2}^l\|\bar E^n-\bar E^{n-1}\|_h^2,
 \\
|r_2(\Phi)|\le& \frac{1}{8}\sum_{n=2}^l\tau_n \left\| \frac{\bar E^n-\bar E^{n-1}}{\tau_n}\right\|_h^2+ch^4\max_{1\le n \le N}\|y^n\|_{H^2(\Gamma_0)}, \\
|r_3(\Phi)| \le& \frac{1}{8}\sum_{n=2}^l\tau_n \left\| \frac{\bar E^n-\bar E^{n-1}}{\tau_n}\right\|_h^2
+ch^2\int_0^T\|y_t\|^2_{H^1(\Gamma_0)}dt, \\
L_2=& L_4=L_5=L_6=0, \\
|L_1+L_3| \le& O(h^2)\left(\|f\|_{L^2(0,T; L^2(\Gamma_0))}^2+\sum_{n=1}^N\tau_n \left\| \frac{\bar E^n-\bar E^{n-1}}{\tau_n}\right\|_h^2\right), \\
|L_7 &+ ... + L_{10}| \\
\le& (O(\tau)+O(h^2))\left(\|y\|^2_{L^2(0, T; H^1(\Gamma_0))}+\sum_{n=1}^N\tau_n\left\|\frac{\bar E^n-\bar E^{n-1}}{\tau_n}\right\|_h^2\right), \\
|L_{11}+L_{12}|\le& O(h^2)\left(\|y_t\|^2_{L^2(0, T; L^2(\Gamma_0))}+\sum_{n=2}^N\tau_n \left\|\frac{\bar E^n-\bar E^{n-1}}{\tau_n} \right\|_h^2\right).  
\end{aligned}
\end{equation}

We estimate $\|E^1\|_h^2$ and $\|\nabla E^1\|_{\tilde g(t_1)}^2$. We choose $\Phi \in W_{h, \tau}$ with $\Phi^n=0$ for $n\ge 2$ and $\Phi^1=\bar E^1$ then (\ref{n12}) implies
\begin{equation} \label{n20}
A(\bar E, \Phi) = A(E, \Phi) *e^{\mu t_n}.
\end{equation}
We estimate the left-hand side of (\ref{n20})  from below
\begin{equation} \label{n21}
\begin{aligned}
A(\bar E, \Phi) =& \tau_1(\nabla \bar E^1, \nabla \Phi^1)_{\tilde g(t_1)} + (\bar E^1, \Phi^1)_h + \tau_1(b^i\nabla_i \bar E^1, \Phi^1)_h\\
&+\tau_1 (c(t_1)\bar E^1, \Phi^1)_h \\
\ge& \frac{\tau_1}{2}\|\nabla \bar E^1\|^2_{\tilde g(t_1)} + \frac{1}{2}\|E^1\|_h^2.
\end{aligned}
\end{equation}
Now we estimate the right-hand side of (\ref{n20})  from above and show that it is bounded from above by terms of type $c\tau^2 + c h^4$ where $c$ depends on $y$ and terms which can be absorbed by the terms on the right-hand side of (\ref{n21}). W.l.o.g. we may estimate $A(E, \Phi)$  instead of $A(E, \Phi)* e^{\mu t}$ and for this estimate we use (\ref{n4}). 

We begin with the term $r_1(\Phi)$ and present the details only for $J_1$, we have 
\begin{equation}
\begin{aligned}
J_1 =& \tau_1 \left(\nabla \widehat{y^1}-\frac{1}{\tau_1} \int_{t_0}^{t_1}\nabla \hat y dt, \nabla \Phi ^1\right)_{\tilde g(t_1)} \\
\le& \tau_1 \left\| \frac{1}{\tau_1}\int_{t_0}^{t_1}\int_t^{\xi_t}|\nabla y_t|\right\|_h\|\nabla \Phi^1\|_h \\
\le& c \tau^{\frac{3}{2}}\|\nabla y_t\|_{L^2(0,T, L^2(\Gamma_0))} \|\nabla \Phi^1\|_h \\
\le& c\tau^{\frac{3}{2}} \left(\frac{\epsilon}{\tau^{\frac{1}{2}}}\|\nabla \Phi^1\|_h^2+\frac{\tau^{\frac{1}{2}}}{\epsilon}\|\nabla y_t\|_{L^2(0,T, L^2(\Gamma_0))}^2\right).
\end{aligned}
\end{equation}
Furthermore, we have
\begin{equation}
r_2(\Phi) = \tau_1 (\widehat{y^1}-R^1_h y ^1, \Phi^1)_h = J_{13} \le c \tau h^2 \max_{t_0\le t \le t_1}\|y(t)\|,^2_{H^2(\Gamma_0)}
\end{equation}
\begin{equation}
r_3(\Phi) = J_{17}+J_{18} =0
\end{equation}
and
\begin{equation}
\begin{aligned}
r_4(\Phi) =& (R^1_h y^1-\widehat{y ^1}, \Phi^1)_h \\
\le& c h^2  \max_{t_0\le t \le t_1}\|y(t)\|^2_{H^2(\Gamma_0)}\|\Phi^1\|_h \\
\le& c h^2\left(\frac{\epsilon}{h^2}\|\Phi^1\|_h^2+\frac{h^2}{\epsilon} \max_{t_0\le t \le t_1}\|y(t)\|^2_{H^2(\Gamma_0)}\right).
\end{aligned}
\end{equation}
There holds $L_{11}=L_{12}=0$. The term $L_7$ together with the first summand in $L_{10}$ can be estimated from above by 
\begin{equation}
(O(h^2)+O(\tau))\|\nabla y\|_{L^2(0, T; L^2(\Gamma_0))}\|\bar E^1\|_h
\end{equation}
which is sufficient and the remaining terms can be treated similarly.
 \end{proof}

\end{document}